\documentclass[12pt]{amsart}
\usepackage{amssymb,slashed, color,array}
\usepackage{amsmath,amsfonts,amsthm,euscript,mathrsfs,multirow}
\usepackage[all]{xy}
\usepackage[english]{babel}
\usepackage{epsf}
\usepackage{graphicx}
\csname @addtoreset\endcsname{equation}{section}

\newtheorem{theorem}{Theorem}[section]
\newtheorem{lemma}[theorem]{Lemma}

\newtheorem{claim}[theorem]{Claim}

\theoremstyle{definition}
\newtheorem{definition}[theorem]{Definition}
\theoremstyle{remark}
\newtheorem{remark}[theorem]{Remark}
\newtheorem{example}[theorem]{Example}

\newenvironment{notation and conventions}{\textbf{Notation and conventions.}}{ }
\DeclareFontFamily{U}{rsf}{} \DeclareFontShape{U}{rsf}{m}{n}{ <5> <6> rsfs5 <7> <8> <9> rsfs7 <10-> rsfs10}{}
\DeclareMathAlphabet\Scr{U}{rsf}{m}{n}
\definecolor{pink}{rgb}{1,0,1}

  \usepackage[pdftex]{hyperref}

\begin{document}

\begin{center}
\baselineskip=14pt{\LARGE
 On Milnor classes via invariants of singular subschemes\\
}
\vspace{1.5cm}
{\large  James Fullwood$^{\spadesuit}$
  } \\
\vspace{.6 cm}

${}^\spadesuit$Institute of Mathematical Research, The University of Hong Kong, Pok Fu Lam Road, Hong Kong.\\

\end{center}

\vspace{1cm}
\begin{center}

{\bf Abstract}
\vspace{.3 cm}
\end{center}

{\small
We derive a formula for the Milnor class of scheme-theoretic global complete intersections (with arbitrary singularities) in a smooth variety in terms of the Segre class of its singular scheme. In codimension one the formula recovers a formula of Aluffi for the Milnor class of a hypersurface.  }

 \tableofcontents{}






\section{Introduction}\label{intro}
Milnor classes are a generalization (at the level of classes in a Chow group) of a numerical invariant John Milnor associated with an isolated singularity of a complex hypersurface in his seminal monograph``Singular points of complex hypersurfaces'' \cite{Milnor}. More precisely, given a (possibly singular) subscheme $X$ of a smooth ambient variety $M$ (proper over an algebraically closed field of characterstic zero) its Milnor class is an element of its Chow group supported on the singular locus of $X$ which we denote by $\mathcal{M}(X)$, and is defined as\footnote{At the moment we blindly ignore any sign conventions one may associate with this class.}

\[
\mathcal{M}(X):=c_{\text{SM}}(X)-c_{\text{FJ}}(X),
\]
\\
where $c_{\text{SM}}(X)$ denotes the \emph{Chern-Schwartz-MacPherson (or simply CSM) class} of $X$ and $c_{\text{FJ}}(X)$ denotes the \emph{Fulton-Johnson class} of $X$ (for the uninitiated we recommend \cite{CharacteristicSingularVariety}). Both Chern-Schwartz-MacPherson and Fulton-Johnson classes are generalizations of Chern classes to the realm of singular varieties, and as such specialize to the total homology Chern class in the case that $X$ is smooth. For $X$ a complete intersection in some smooth ambient variety $M$, the Fulton-Johnson class coincides with a `canonical class' for singular varieties defined by William Fulton \cite{IntersectionTheory} for any subscheme of a smooth variety, which we refer to as the \emph{Fulton class}. As we will restrict our attention to Milnor classes of complete intersections in a smooth variety $M$, in the definition of Milnor class we may replace the Fulton-Johnson class of $X$ by the Fulton class of $X$, which we denote by $c_{\text{F}}(X)$. The justification of the moniker ``Milnor class" is that if $X$ is a hypersurface with isolated singularities then (up to sign)

\[
\mathcal{M}(X)= ``\text{the sum of the Milnor numbers over each singular point of $X$}",
\]
\\
and thus captures the essence of `Milnor number' on a global level. Milnor classes are then a vast generalization of Milnor's invariant, as they exist for arbitrary singularities and reside in a Chow group. As the Fulton class of $X$ coincides with the total Chern class of a smooth variety in the same rational equivalence class as $X$, we may view the Milnor class (of a complete intersection) as measuring the difference between $c_{\text{SM}}(X)$ and the Chern class of a smooth deformation of $X$ (parametrized by $\mathbb{P}^1$). Numerically speaking, since (over $\mathbb{C}$) $\int_X c_{\text{SM}}(X)=\chi_{\text{top}}(X)$ , integration of the Milnor class of $X$ measures the deviation of the (topological) Euler characteristic of $X$ from that of a smooth deformation. However, topological characterizations of the higher dimensional Milnor classes remain elusive at best.

In recent years much work has gone into the investigation of such classes and as such much progress in our understanding of these (at least from a computational perspective) has been made (e.g., \cite{MilnorMaxim}\cite{MR1873009}\cite{MR1902644}\cite{MR3053711}\cite{MCLCI}). As many insights have come predominantly from a topological/complex-analytic perspective (e.g. in terms of the geometry of a Whitney stratification of $X$), we adopt the perspective of Aluffi and seek a purely algebraic characterization in terms of (intersection-theoretic) invariants of a natural scheme structure on the singular locus of $X$. For $X$ a hypersurface Aluffi has proved\footnote{This is actually a formula for $i_*\mathcal{M}(X)$, where $i:X\hookrightarrow M$ is the inclusion. Moreover, here and throughout we omit pushforwards (and pullbacks) via inclusions.} \cite{MR1697199}

\begin{equation}
\mathcal{M}(X)=\frac{c(TM)}{c(\mathscr{O}(X))}\cap \left(s(Y,M)^{\vee} \otimes_M \mathscr{O}(X)\right),
\end{equation}
\\
where `$^{\vee}$' and `$\otimes_M$' here are intersection-theoretic operations which we recall in section \S\ref{td}, $s(Y,M)$ denotes the \emph{Segre class}\footnote{We recommend either of \cite{CharacteristicSingularVariety}\cite{IntersectionTheory} for a nice introduction to Segre classes.} of the singular \emph{subscheme} $Y$ of $X$ in $M$ and by `singular subscheme' we mean the subscheme of $X$ whose ideal sheaf is the restriction to $X$ of the ideal sheaf over $M$ which is locally generated by a defining equation for $X$ and each of its corresponding partial derivatives. Thus from an algebraic perspective, the Milnor class of a hypersurface is in essence captured by the Segre class of its singular scheme.

In the closing line of \cite{MR2007377}, Aluffi states ``...It is not even known whether $m(X,M)$--and hence the Milnor class of $X$--is determined by the singular subscheme of $X$, \emph{even when $X$ is a complete intersection of codimension 2.}'' ($m(X,M)$ is the unique class such that $\mathcal{M}(X)=c(TM)\cap m(X,M)$). Our aim in this note is to provide a partial response to this inquiry for every codimension.

So let $M$ be a smooth algebraic variety over an algebraically closed field of characteristic zero, and let $X$ be a \emph{global} complete intersection corresponding to the zero-scheme of a section of a vector bundle $\mathscr{E}\to M$ (note that our assumption that $X$ is a global complete intersection implies that $\mathscr{E}$ splits). Our situation is constrained by the fact that we assume that if $X$ is cut out by $k=\text{rk}(\mathscr{E})$ hypersurfaces $M_1, \ldots, M_k$, then $M_{_1}\cap \cdots \cap M_{k-1}$ is (scheme-theoretically) smooth\footnote{The strength of this assumption seems to grow with the codimension of $X$.} (surely the ordering of the indicies is irrelevant here). Let $\mathscr{L}\to X$ denote the line bundle associated with the divisor $M_k$, and denote by $Y$ the singular (sub)scheme of $X$ (whose precise definition is given in \S\ref{proof}). Under these assumptions, our result is the following\footnote{The formula involves two notions of `$\otimes$': the traditional tensor product of vector bundles which we denote by `$\otimes$', and an intersection-theoretic operation `$\otimes_M$', which we define in section \S\ref{td}.}

\begin{theorem}\label{mt} Let $X$ be a global complete intersection corresponding to the zero-scheme of a section of a vector bundle $\mathscr{E}\to M$, subject to the assumptions above. Then

\begin{equation}
\mathcal{M}(X)=\frac{c(TM)}{c(\mathscr{E})}\cap \left(c(\mathscr{E^{\vee}\otimes \mathscr{L}})\cap (s(Y,M)^{\vee}\otimes_M \mathscr{L}) \right).
\end{equation}
\\
\end{theorem}
In the hypersurface case, $X$ is the zero-scheme of a rank one bundle $\mathscr{E}$, with $\mathscr{L}=\mathscr{E}=\mathscr{O}(X)$, and thus (1.2) recovers the formula (1.1) of Aluffi in codimension one. Surely the assumption that $M_1\cap \cdots \cap M_{k-1}$ is smooth prohibits this formula from being representative at the level of full generality, as it forces the formula to depend lopsidedly on $\mathscr{L}$ compared to the line bundles corresponding to the other hypersurfaces which cut out $X$. In any case, we note that the utility formula (1.2) is at least two-fold. Not only only does such a formula give a precise characterization of the Milnor class of a complete intersection in terms of its singular scheme, the formula may be `inverted' to yield a formula for the Segre class of its singular scheme in terms of its Milnor class. As Segre classes are in general very difficult to compute directly from their definition, alternative means of computing Segre classes are very desirable, especially in the context of enumerative geometry. As an illustration, we compute a non-trivial Segre class in terms of Milnor classes in section \S\ref{ex}.

In what follows we review the intersection-theoretic calculus of the operations `$^{\vee}$' and `$\otimes_M$', prove the theorem, then give some examples and applications.
\\

\emph{Acknowledgements.} The author would like to thank Paolo Aluffi for useful discussions throughout the course of this project, which not only identified several inaccuracies but were instrumental to its completion.

\section{The `tensorial' and `dual' operations}\label{td} In \cite{MR1316973}, two intersection-theoretic operations on classes in a Chow group were introduced which not only streamline many intensive computations, but also often provide compact ways to write seemingly complicated formulas. In particular, let $M$ be a variety and denote its Chow group by $A_*M$. We write a class $\alpha\in A_*M$ as $\alpha=\alpha^0+\cdots +\alpha^n$, where $\alpha^i$ is the component of $\alpha$ of \emph{codimension} $i$ (in $M$). We denote by $\alpha^{\vee}$ the class\footnote{We note that the map $\alpha\mapsto \alpha^{\vee}$ coincides with the map $\tau$ defined in \cite{MR2141853}.}
\[
\alpha^{\vee}:=\sum (-1)^i \alpha^i.
\]
The justification for the `dual' notation is straightforward: If $\mathscr{E}\to M$ is a vector bundle with total Chern class $c(\mathscr{E})=\sum c_i(\mathscr{E})$ then $c(\mathscr{E}^{\vee})=\sum (-1)^ic_i(\mathscr{E})=c(\mathscr{E})^{\vee}$. Next, we introduce an action of the Picard group of $M$ on $A_*M$. Given a line bundle $\mathscr{L}\to M\in \text{Pic}(M)$ we define its action on $\alpha=\sum \alpha^i\in A_*M$ as

\[
\alpha \otimes_M \mathscr{L}:= \sum \frac{\alpha^i}{c(\mathscr{L})^i}.
\]
It is also straightforward to see that this honestly defines an action of $\text{Pic}(M)$ on $A_*M$ \cite{MR1316973} (i.e., $(\alpha \otimes_M \mathscr{L})\otimes_M \mathscr{M}=\alpha \otimes_M (\mathscr{L}\otimes \mathscr{M})$) . These innocuous definitions lend their utility throughout this note via the following two formulas:

\begin{equation}
\left(c(\mathscr{E})\cap \alpha\right)^{\vee}=c(\mathscr{E}^{\vee})\cap \alpha^{\vee},
\end{equation}
\begin{equation}
\left(c(\mathscr{E})\cap \alpha\right) \otimes_M \mathscr{L}=\frac{c(\mathscr{E}\otimes \mathscr{L})}{c(\mathscr{L})^r}\cap \left(\alpha \otimes_M \mathscr{L}\right),
\end{equation}
\\
where $r$ is the rank of $\mathscr{E}$ in the Grothendieck group of vector bundles on $M$. The proofs of these formulas may be found in \cite{MR1316973}, which follow directly from the definitions. As an illustration of how these operations may `compactify' a cumbersome formula, consider the simple expression for the class

\[
s(X\setminus Y):=\frac{1}{c(\mathscr{O}(X))}\cap \left(s(Y,M)^{\vee} \otimes_M \mathscr{O}(X)\right),
\]
\\
which appears in Aluffi's hypersurface formula (1.1). Without the `tensor' and `dual' operations, the most efficient way of a giving a formula for $s(X\setminus Y)$ is to give a formula for its component of  dimension $m$, which reads \cite{MR1316973}

\[
s(X\setminus Y)_m=s(X,M)_m+(-1)^{n-m}\sum_{j=0}^{n-m}\left(\begin{array}{c} n-m \\ j \end{array}\right)X^j\cdot s(Y,M)_{m+j},
\]
\\
where $n$ is the dimension of $M$ and by $X^j$ we mean the $j$-fold intersection product of (the divisor class associated with) $X$ with itself. So not only do the tensor and dual operations dispense of the appearance of complicated summations involving binomial coefficients, they provide a means of succinctly capturing all components of $s(X\setminus Y)$ at once. Moreover, computations throughout the rest of this note will serve as illustrations of their computational utility.

We conclude this section with a lemma needed for the proof of Theorem \ref{mt} (we omit the proof as it follows immediately from the definition of the tensorial operation).

\begin{lemma}\label{tensor} Let $M$ be a variety, $M'\overset{i} \hookrightarrow M$ be a regular embedding of codimension $d$, $\alpha\in A_*M'$ and let $\mathscr{L}\to M$ be a line bundle. Then

\begin{equation}
\alpha\otimes_{M'} i^*\mathscr{L} =c(\mathscr{L})^d \cap \left(\alpha\otimes_M \mathscr{L}\right).
\end{equation}

\end{lemma}

\section{The proof(s)}\label{proof}

We now provide a proof of Theorem \ref{mt}, along with a sketch of an alternative proof in codimension two which may shed light on the more general case of an arbitrary global complete intersection.

Denote by $M$ a smooth ambient variety and let $X$ be a (global) complete intersection (in $M$) corresponding to the zero-scheme of a section of a vector bundle $\mathscr{E}\to M$ whose rank we denote by $k$. We assume that there exists a smooth complete intersection $Z$ in $M$ of codimension $k-1$ such that $X=Z\cap M_k$ (scheme-theoretically), where $M_k$ is a hypersurface in $M$. Denote the line bundle corresponding to $M_k$ by $\mathscr{L}$ and let $Y$ denote the \emph{singular scheme} of $X$, which we now define:

\begin{definition} Let $M$ be smooth of dimension $n$ and let $X\subset M$ be a (local) complete intersection of codimension $k$. Then we define the \emph{singular (sub)scheme of} $X$ to be the subscheme whose ideal sheaf over $X$ is the restriction to $X$ of the ideal sheaf over $M$ which is locally generated by a set of local defining equations ($F_1(x_1,\ldots,x_n)=\cdots=F_k(x_1,\ldots ,x_n)=0$) of $X$ along with the $k\times k$ minors of the corresponding matrix of partial derivatives  $a_{ij}=\frac{\partial F_i}{\partial x_j}$ (i.e., the $k$th \emph{Fitting ideal} of the coordinate ring of the corresponding affine open subscheme of $X$).
\end{definition}

\begin{remark} As $X$ may often be embedded in different smooth varieties as a complete intersection, there arises an issue as to whether the notion of singular scheme given above is well defined. However, it follows from general results in the theory of Fitting ideals that indeed it is \cite{MR1322960}.
\end{remark}

\begin{proof}[Proof of Theorem \ref{mt}]. By assumption $X$ is a hypersurface in $Z$, thus Aluffi's hypersurface formula (1.1) yields

\begin{equation}
\mathcal{M}(X)=\frac{c(TZ)}{c(\mathscr{L})}\cap \left(s(Y,Z)^{\vee}\otimes_Z \mathscr{L}\right).
\end{equation}
\\
Since $X$ is also a complete intersection in $M$,

\[
\frac{c(TM)}{c(\mathscr{E})}=c(TX_{\text{vir}})=\frac{c(TZ)}{c(\mathscr{L})},
\]
where $TX_{\text{vir}}$ denotes the virtual tangent bundle of $X$. Moreover, since $Y$ is a subscheme of $Z$, which in turn is a smooth subvariety of $M$, $Y$ is linearly embedded in $M$ \cite{MR1040263}, so

\[
s(Y,Z)=c(N_ZM)\cap s(Y,M),
\]
\\
where $N_ZM$ denotes the normal bundle to $Z$ in $M$. Formula (2.1) then yields

\[
s(Y,Z)^{\vee}=c(N_ZM^{\vee})\cap s(Y,M)^{\vee}.
\]
\\
Thus\footnote{As a consequence of taking duals in $M$ rather than $Z$, the second equality should yield a factor of $(-1)^{k-1}$, which we omit due to the absence of a standard sign convention associated with $\mathcal{M}(X)$ in the literature. We do however envoke a sign convention in \S\ref{ex} that accounts for the `missing' factor of $(-1)^{k-1}$.}

\begin{eqnarray*}
\mathcal{M}(X)&=&\frac{c(TM)}{c(\mathscr{E})}\cap \left((c(N_ZM^\vee)\cap s(Y,M)^{\vee})\otimes_Z \mathscr{L}\right) \\
                        &\overset{(2.3)}=&\frac{c(TM)c(\mathscr{L})^{k-1}}{c(\mathscr{E})}\cap \left((c(N_ZM)^{\vee}\cap s(Y,M)^{\vee})\otimes_M \mathscr{L}\right) \\
                        &\overset{(2.2)}=&\frac{c(TM)c(\mathscr{L})^{k-1}}{c(\mathscr{E})}\cap \left(\frac{c(N_ZM^{\vee}\otimes \mathscr{L})}{c(\mathscr{L})^{k-1}}\cap (s(Y,M)^{\vee}\otimes_M \mathscr{L})\right)\\
                        &=&\frac{c(TM)}{c(\mathscr{E})}\cap \left(c(\mathscr{E}^{\vee}\otimes \mathscr{L})\cap(s(Y,M)^{\vee}\otimes_M \mathscr{L})\right).
\end{eqnarray*}
\\
To arrive at the last equality we cancelled a common factor of $c(\mathscr{L})^{k-1}$ from the numerator and denominator of the previous expression, and then we used the fact that since $\mathscr{E}=N_ZM\oplus \mathscr{L}$, $\mathscr{E}^{\vee}\otimes \mathscr{L}=(N_ZM^{\vee}\otimes \mathscr{L})\oplus \mathscr{O}$, thus

\[
c(\mathscr{E}^{\vee}\otimes \mathscr{L})=c(N_ZM^{\vee}\otimes \mathscr{L}),
\]
\\
concluding the proof.
\end{proof}

We now sketch an alternative proof in codimension two, as we feel it may lead to a proof of the general case in codimension two (i.e. when \emph{both} hypersurfaces cutting out $X$ are possibly singular), which in turn may yield the correct form for a general formula in arbitrary codimension. Before doing so we need the following \cite{MR2007377}

\begin{definition}
Let $X$ be a hypersurface in some smooth ambient variety $M$ and denote its singular scheme by $Y$. We define the \emph{SM-Segre class} of $X$ to be the class

\[
s^{\circ}(X,M):=s(X,M)+c(\mathscr{O}(X))^{-1}\cap \left(s(Y,M)^{\vee}\otimes_M \mathscr{O}(X)\right)\in A_*M.
\]
\end{definition}

Note that by formula (1.1) $\mathcal{M}(X)=c(TM)\cap \left(s^{\circ}(X,M)-s(X,M)\right)$. Another way of saying this is that $c_{\text{SM}}(X)=c(TM)\cap s^{\circ}(X,M)$, so by inclusion-exclusion for Chern-Schwartz-MacPherson classes (3.2) we may now inductively define $s^{\circ}(X,M)$ for $X$ a global complete intersection.

\begin{definition}
Let $X$ be a global complete intersection of codimension $k$ in some smooth variety $M$, and let $M_1,\ldots, M_k$ be hypersurfaces such that $X=M_1\cap \cdots \cap M_k$ (scheme-theoretically). Then we define the \emph{SM-Segre class} of $X$ to be the class

\[
s^{\circ}(X,M):=\sum_{s=1}^{k}(-1)^{s-1}\left(\sum_{i_1< \cdots <i_s}s^{\circ}(X_{i_1}\cup \cdots \cup X_{i_s},M)\right)\in A_*M.
\]
\end{definition}
 Since by inclusion-exclusion $c_{\text{SM}}(X)=c(TM)\cap s^{\circ}(X,M)$, it follows that this definition is independent of the hypersurfaces chosen to cut out $X$.
 \\

 \begin{proof}[Alternative proof in codimension two (sketch)] We assume here that $X=M_1\cap M_2\subset M$ is a complete intersection of codimension two (the extra assumption in Theorem\ref{mt} in this case requires that one of the $M_i$ be smooth, but at this point we make no smoothness assumptions on the $M_i$), and we denote by $\mathscr{L}_i$ the line bundle corresponding to $M_i$. By definiton of Milnor and SM-Segre classes we have

 \[
 \mathcal{M}(X)=c(TM)\cap \left(s^{\circ}(X,M)-s(X,M)\right).
 \]
 \\
 We now compute $s^{\circ}(X,M)-s(X,M)$:

 It follows from Theorem 1.1 in \cite{MR1986112} that

 \begin{eqnarray*}
 s(X,M)&=&s(M_1,M)+s(M_2,M)-s(M_1\cup M_2)- \\
       & &c(\mathscr{L}_1\otimes \mathscr{L}_2)^{-1}\cap (s(X,M)^{\vee}\otimes_M \mathscr{L}_1\otimes \mathscr{L}_2), \\
 \end{eqnarray*}
and from the definition of SM-Segre classes we have

\begin{eqnarray*}
s^{\circ}(X,M)&=&s^{\circ}(M_1,M)+s^{\circ}(M_2,M)-s^{\circ}(M_1\cup M_2,M) \\
              &=&\left(s(M_1,M)+c(\mathscr{L}_1)^{-1}\cap (s(Y_1,M)^{\vee}\otimes_M \mathscr{L}_1)\right) \\
              & &+\left(s(M_2,M)+c(\mathscr{L}_2)^{-1}\cap (s(Y_2,M)^{\vee}\otimes_M \mathscr{L}_2)\right) \\
              & & -\left(s(M_1\cup M_2)+c(\mathscr{L}_1\otimes \mathscr{L}_2)^{-1}\cap
              \left(s(\overline{X},M)^{\vee}\otimes_M \mathscr{L}_1\otimes \mathscr{L}_2\right)\right), \\
\end{eqnarray*}
\\
where $Y_i$ denotes the singular scheme of the possibly singular hypersurface $M_i$ and $\overline{X}$ denotes the singular scheme of $M_1\cup M_2$ (note that $X$ is a subscheme of $\overline{X}$). We then have

\begin{eqnarray*}
s^{\circ}(X,M)-s(X,M)&=&c(\mathscr{L}_1)^{-1}\cap \left(s(Y_1,M)^{\vee}\otimes_M \mathscr{L}_1\right)+c(\mathscr{L}_2)^{-1}\cap \left(s(Y_2,M)^{\vee}\otimes_M \mathscr{L}_2\right) \\
                     & &-c(\mathscr{L}_1\otimes \mathscr{L}_2)^{-1}\cap \left((s(\overline{X},M)-s(X,M))^{\vee}\otimes_M \mathscr{L}_1\otimes \mathscr{L}_2\right) \\
\end{eqnarray*}
Now if one of the $M_i$ are smooth, say $M_1$, then the term $c(\mathscr{L}_1)^{-1}\cap \left(s(Y_1,M)^{\vee}\otimes_M \mathscr{L}_1\right)$ vanishes in the previous equation above, and then using the inclusion-exclusion formula for CSM classes, i.e.,

\begin{equation}
c_{\text{SM}}(M_1\cup M_2)=c_{\text{SM}}(M_1)+c_{\text{SM}}(M_2)-c_{\text{SM}}(M_1\cap M_2),
\end{equation}
\\
one may (after judicious use of the tensorial and dual operations, and then using the smoothness assumption on $M_1$ to compute $c_{\text{SM}}(M_1\cap M_2)$) solve for $(s(\overline{X},M)-s(X,M))^{\vee}$ which appears in the previous equation for $s^{\circ}(X,M)-s(X,M)$, the result of which yields the conclusion of Theorem \ref{mt} in the codimension two case.
\end{proof}

What we find somewhat surprising about this proof is that the expression derived above for $s^{\circ}(X,M)-s(X,M)$ depends solely on the the Segre classes of the singular schemes of $M_1$, $M_2$, $M_1\cup M_2$ and the Segre class of $X$. But after assuming that $M_1$ is smooth and then plugging in our computation of $(s(\overline{X},M)-s(X,M))^{\vee}$ into the formula we derived above for $s^{\circ}(X,M)-s(X,M)$, the term $c(\mathscr{L}_2)^{-1}\cap \left(s(Y_2,M)^{\vee}\otimes_M \mathscr{L}_2\right)$ cancels and all that remains is an expression depending on the Segre class of $Y$ (the singular scheme of $X$). We then naturally suspect that if one may compute $(s(\overline{X},M)-s(X,M))^{\vee}$ without any smoothness assumptions on $M_1$ or $M_2$, the result of which would cancel \emph{both} the contributions of the $c(\mathscr{L}_i)^{-1}\cap \left(s(Y_i,M)^{\vee}\otimes_M \mathscr{L}_i\right)$ in the formula derived for $s^{\circ}(X,M)-s(X,M)$, again yielding an expression which depends only on the singular scheme of $X$. Unfortunately, a means for such a computation presently eludes us.

\section{An example with application}\label{ex}
We now invoke the following sign convention for Milnor classes:

\[
\mathcal{M}(X):=(-1)^{s}(c_{\text{F}}(X)-c_{\text{SM}}(X)),
\]
\\
where $X$ is a complete intersection in some smooth variety $M$ and we set $s$ equal to the parity of the codimension of $X$ in $M$.

\begin{example} Let $X=Q\cap H$, where $Q:(x_0^2-x_1x_2=0)\subset M=\mathbb{P}^4$ is a singular quadric and $H:(x_0=0)$ is a hyperplane, whose class in $A_*\mathbb{P}^4$ we denote by $H$ as well (thus $[Q]=2H$). Then $X$ is the union of two linear subspaces of codimension two in $\mathbb{P}^4$ which intersect along a linear subspace of codimension three, which is the singular scheme $Y$ of $X$ (the singular scheme is reduced in this case). The normal bundle to $X$ in $\mathbb{P}^4$ is then (the restriction to $X$ of) $\mathscr{E}=\mathscr{O}(1)\oplus \mathscr{O}(2)$, and the dual of the Segre class of  $Y$ in $\mathbb{P}^4$ is

\[
s(Y,M)^{\vee}=\left(\frac{-H}{1-H}\right)^3.
\]
\\
Since $Q$ is singular and $H$ is smooth, $\mathscr{L}$ in formula (1.2) is necessarily $\mathscr{O}(2)$. Theorem \ref{mt} then yields\footnote{We also use $H$ to denote $c_{1}(\mathscr{O}(1))$.}

\begin{eqnarray*}
\mathcal{M}(X)&=&(-1)\frac{(1+H)^{5}(1+H)}{(1+H)(1+2H)}\cap \left(\left(\frac{-H}{1-H}\right)^3\otimes_{\mathbb{P}^4}\mathscr{O}(2)\right) \\
                        &=&(-1)\frac{(1+H)^{5}}{(1+2H)}\cap \left(\frac{-H}{(1+2H-H)}\right)^3 \\
                       &=&(-1)\frac{(1+H)^{2}}{(1+2H)}\cap (-H^3) \\
                       &=&[\mathbb{P}^1].
\end{eqnarray*}
\\
One may also compute $\mathcal{M}(X)$ in this example `by hand': Since we know that the singular scheme of $X$ is just a (reduced) line in $\mathbb{P}^4$, we know $\mathcal{M}(X)=[\mathbb{P}^1]+n[pt]$ for some integer $n$. But $n$ is necessarily $\chi_{\text{top}}(\tilde{X})-\chi_{\text{top}}(X)$, where $\tilde{X}$ is a smooth representative of the rational equivalence class of $X$, i.e., a smooth quadric surface. As smooth quadrics are isomorphic to $\mathbb{P}^1\times \mathbb{P}^1$, we have $\chi_{\text{top}}(\tilde{X})=2\cdot 2=4$, and by inclusion-exclusion we have $\chi_{\text{top}}(X)=\chi_{\text{top}}(\mathbb{P}^2)+\chi_{\text{top}}(\mathbb{P}^2)-\chi_{\text{top}}(\mathbb{P}^1)=3+3-2=4$,
thus $n=0$ and $\mathcal{M}(X)=[\mathbb{P}^1]$, as given above by Theorem \ref{mt}.
\end{example}

As mentioned in \S\ref{intro}, formulas for Milnor classes in terms of Segre classes of singular schemes may be used to compute Segre classes, which in general are very difficult to compute from their definition alone. As an illustration we use the previous example to compute the Segre class of the singular scheme of the hypersurface $Z$ which is the union of $Q$ and $H$ as given in Example 4.1., i.e.

\[
Z:(x_0^3-x_0x_1x_2=0)\subset \mathbb{P}^4.
\]
\\
Computing partial derivatives we see the the singular scheme of $Z$, which we denote by $Z_s$, is the subscheme of $Z$ corresponding to the (homogeneous) ideal

\[
 I=(3x_0^2-x_1x_2, x_0x_2, x_0x_1).
 \]
 \\
We now compute $s(Z_s,\mathbb{P}^4)$ by first inverting the formula for the Milnor class of $Z$, which we then relate to the CSM classes of $Q$, $H$ and $X=Q\cap H$ (which all have simple singular schemes), the details of which are given via the proof of the following

\begin{claim} Let $Z_s$ be the singular scheme of the hypersurface $Z$ defined above. Then

\[
s(Z_s,\mathbb{P}^4)=2[\mathbb{P}^2]-4[\mathbb{P}^1]
\]
\\
\end{claim}

\begin{proof}
By formula (1.1) we have

\[
\mathcal{M}(Z)=\frac{c(T\mathbb{P}^4)}{c(\mathscr{O}(3))}\cap \left(s(Z_s,\mathbb{P}^4)^{\vee}\otimes_{\mathbb{P}^4} \mathscr{O}(3)\right),
\]
\\
thus

\begin{equation}
\frac{c(\mathscr{O}(3))}{c(T\mathbb{P}^4)}\cap \mathcal{M}(Z)=s(Z_s,\mathbb{P}^4)^{\vee}\otimes_{\mathbb{P}^4} \mathscr{O}(3).
\end{equation}
\\
Tensoring both sides of (4.1) by $\mathscr{O}(-3)$ and then taking duals we get

\begin{equation}
s(Z_s,\mathbb{P}^4)=\left(\frac{c(\mathscr{O}(3))}{c(T\mathbb{P}^4)}\cap \mathcal{M}(Z)\right)^{\vee}\otimes_{\mathbb{P}^4} \mathscr{O}(3).
\end{equation}
\\
Now by inclusion-exclusion for CSM classes along with the fact that $c_{\text{SM}}(Z)=c_{\text{F}}(Z)+\mathcal{M}(Z)$ we have

\[
\mathcal{M}(Z)=c_{\text{SM}}(Q)+c_{\text{SM}}(H)-c_{\text{SM}}(X)-c_{\text{F}}(Z).
\]
\\
As the singular scheme of $Q$ is the line $l:(x_0=x_1=x_2=0)\subset \mathbb{P}^4$ (which is the same as the singular scheme of $X$), whose (dual) Segre class was computed in Example 4.1, its CSM class is easily computed to be

\[
c_{\text{SM}}(Q)=\frac{(1+H)^5\cdot 2H-(1+H)^2\cdot H^3}{1+2H}.
\]
\\
Moreover\footnote{We apologize here for denoting both the hyperplane $\{x_0=0\}$ and $c_1(\mathscr{O}(1))$ by $H$.}

\[
c_{\text{SM}}(H)=\frac{(1+H)^5\cdot H}{1+H}, \quad c_{\text{F}}(Z)=\frac{(1+H)^5\cdot 3H}{1+3H},
\]
\\
and by adding $c_{\text{F}}(X)=\frac{(1+H)^5\cdot 2H^2}{(1+H)(1+2H)}$ to $\mathcal{M}(X)$ (which was computed in Example (4.1)) we have

\[
c_{\text{SM}}(X)=\frac{(1+H)^5\cdot 2H^2+(1+H)^3\cdot H^3}{(1+H)(1+2H)}.
\]
\\
Thus\footnote{In particular, this tells us that $\chi_{\text{top}}(Z)=4$.}

\begin{eqnarray*}
\mathcal{M}(Z)&=& \frac{(1+H)^2(1-H)\cdot 2H^2}{(1+3H)} \\
              &=&2[\mathbb{P}^2]-4[\mathbb{P}^1]+10[\mathbb{P}^0].
\end{eqnarray*}
\\
Plugging $\mathcal{M}(Z)$ into equation (4.2) then yields

\[
s(Z_s,\mathbb{P}^4)=2[\mathbb{P}^2]-4[\mathbb{P}^1],
\]
\\
as desired.

\end{proof}

\bibliographystyle{plain}
\bibliography{mybib}

\end{document}